\documentclass[11pt, a4paper]{amsart}
\usepackage{amssymb}
\usepackage{amsthm}
\usepackage{amsmath}
\usepackage{graphicx}
\usepackage[a4paper, total={170mm,257mm}, left=20mm, top=20mm]{geometry}
\usepackage{enumitem}
\setlist[enumerate, 1]{label={\textup{(\roman*)}}}
\setlist{leftmargin=2em, labelsep=.5em, labelwidth=1.5em, itemindent=0em}
\usepackage[unicode=true, breaklinks=true]{hyperref}
\hypersetup{
    pagebackref  = true,
	colorlinks   = false,
	pdfauthor    = {Matěj Doležálek},
	pdftitle     = {Extending bounds on minimal ranks of universal quadratic lattices to larger number fields},
	pdfkeywords  = {universal quadratic lattice, totally real number field, Galois theory, Goursat's lemma},
	pdfcreator   = {Pdflatex},
	pdfpagemode  = UseNone,
	pdfstartview = FitH
}

\newcommand{\comment}[1]{}

\newtheorem{theorem}{Theorem}[section]
\newtheorem{prop}[theorem]{Proposition}
\newtheorem{lemma}[theorem]{Lemma}
\newtheorem{corollary}[theorem]{Corollary}
\theoremstyle{definition}
\newtheorem{definition}[theorem]{Definition}
\newtheorem{example}[theorem]{Example}

\def\uv#1{``#1''}

\let\zet\Z
\newcommand{\Q}{\mathbb{Q}}\let\kve\Q
\let\er\R
\let\ce\C

\let\phi\varphi
\def\restr#1#2{\left.#1\right|_{#2}}
\DeclareMathOperator{\disc}{disc}
\DeclareMathOperator{\Gal}{Gal}
\DeclareMathOperator{\tr}{Tr}\let\Tr\tr
\def\zav#1{\left(#1\right)}
\def\set#1{\left\{#1\right\}}
\def\abs#1{\left|#1\right|}
\def\Of{\mathcal{O}_F}
\def\Ol{\mathcal{O}_L}
\DeclareMathOperator{\spn}{span}

\usepackage{tikz}
\newenvironment{diagram}[1][]{
    \displaymath\tikzpicture[line width=.6pt, n/.style = {inner sep=.3em}, rotate=45, #1]
}{
    \endtikzpicture\enddisplaymath
}

\title[Extending bounds on minimal ranks of universal lattices]{Extending bounds on minimal ranks of universal quadratic lattices to larger number fields}
\author{Mat\v ej Dole\v z\'alek}
\address{Charles University, Faculty of Mathematics and Physics, Department of Algebra, Sokolov\-sk\' a 83, 18600 Praha~8, Czech Republic}
\email{matej@gimli.ms.mff.cuni.cz}

\thanks{The author acknowledges support by Charles University project GAUK No. 134824.}


\begin{document}

\begin{abstract}
There exist numerous results in the literature proving that within certain families of totally real number fields, the minimal rank of a universal quadratic lattice over such a field can be arbitrarily large. Kala introduced a technique of extending such results to larger fields -- e.g. from quadratic fields to fields of arbitrary even degree -- under some conditions. We present improvements to this technique by investigating the structure of subfields within composita of number fields, using basic Galois theory to translate this into a group-theoretic problem.
In particular, we show that if totally real number fields with minimal rank of a universal lattice $\geq r$ exist in degree $d$, then they also exist in degree $kd$ for all $k\geq3$.
\end{abstract}
\maketitle

\section{Introduction}
\label{sec:intro}

The quest to understand quadratic forms spans much of the history of number theory. This started with early results of Fermat, Lagrange and Legendre, who studied the expressibility of positive integers as sums of two, four and three squares respectively, and continued throughout the following centuries. A~particular interest has been directed at \emph{universal} quadratic forms, i.e. positive definite forms that represent all positive integers. This culminated with Bhargava and Hanke \cite{bhargava-hanke} finding all universal quaternary forms using their 290-theorem, which surprisingly states that a quadratic form is universal if and only if it represents all 29 of the so-called \emph{critical integers}, the largest of which is $290$.

While this provides a satisfying resolution over $\zet$, the situation is much less well-understood over number fields. Here, the usual setting is to consider a totally real number field and a form which attains totally positive values outside of the origin. A~\emph{universal} form is then one which represents all totally positive algebraic integers of the field. One may also slightly generalize quadratic forms to \emph{quadratic lattices}, this essentially corresponds to considering arbitrary finitely generated modules equipped with quadratic forms instead of only free modules (see Subsection~\ref{subsec:fields-and-lattices} for a precise definition).

Some of the early ventures into quadratic forms over number fields include the study of sums of squares over $\kve(\sqrt5)$ by Kirmse \cite{kirmse} and Götzky \cite{gotzky}, over $\kve(\sqrt2)$ and $\kve(\sqrt3)$ by Cohn \cite{cohn} and over arbitrary totally real number fields by Siegel \cite{siegel}, who proved that $\kve$ and $\kve(\sqrt5)$ are the only totally real fields where a sum of squares is universal. Surprisingly, while $\kve$ requires four squares to express every positive integer, over $\kve(\sqrt5)$ already the sum of three squares is universal (this is due to Maaß \cite{maass}). Related to this, the currently still open \emph{Kitaoka's conjecture} posits that there are only finitely many totally real number fields which admit a ternary universal form.

In recent years, indecomposable elements have been used extensively with regards to representation of algebraic integers by quadratic forms over totally real number fields: a totally positive algebraic integer is called \emph{indecomposable} if it may not be written as a sum of two totally positive algebraic integers. These have been use to prove partial results towards Kitaoka's conjecture or to determine which number fields admit a universal form with coefficients from $\zet$ (e.g. \cite{kala-yatsyna1, kala-yatsyna2, krasensky-tinkova-zemkova}).
For further background on techniques related to indecomposable elements as well as on the state of the art regarding Kitaoka's conjecture, we refer the reader to a survey paper by Kala \cite{kala-survey}.

The strand of recent results most pertinent to the present article are the ones which construct fields in a chosen degree such that all universal forms over such a field must have a high number of variables, or more generally such that a universal quadratic lattice must have high rank. Denoting by $m(F)$ the minimal rank of a universal lattice over a totally real number field $F$, fields with $m(F)\geq r$ for arbitrary $r$ are known to exist among quadratic fields \cite{blomer-kala, kalaquad}, cubic fields \cite{kala-tinkova, yatsyna}, biquadratic \cite{cech-et-al} and multiquadratic \cite{kala-svoboda} fields. In fact, later results of Kala, Yatsyna and \.Zmija \cite{kala-yatsyna-zmija} and Man \cite{man} respectively showed that almost all (in a sense of natural density) totally real quadratic and multiquadratic fields have $m(F)$ larger than a given $r$.

Building on this, Kala \cite{kalamain} introduced a technique for extending these results to higher the degrees. Starting from a suitable field $L$ with a high $m(L)$, one finds a field $K$ such that the structure of subfields in the compositum $KL$ is \uv{as simple as possible} (see Subsection~\ref{subsec:construction} for a precise statement). Then, if $\disc_K$ is sufficiently large, it can be shown that all universal lattices over $KL$ in fact contain a universal lattice over $L$, which yields $m(KL)\geq m(L)$, hence $m(KL)$ must also be large. With this, the results cited in the previous paragraph may be used to provide existence of totally real fields $F$ with high $m(F)$ in all degrees divisible by $2$ or $3$.

The original proof in \cite{kalamain} contains several technical assumptions, such as $L$ being Galois over $\kve$ with $\Gal(L/\kve)$ cyclic or the Galois closure $\tilde K$ of $K$ over $\kve$ having $\Gal(\tilde K/\kve)\simeq S_{[K:\kve]}$. Our purpose in this paper is to treat the technique more generally, remove or ease these technical assumptions as much as possible and, in doing so, make it readily applicable if new \uv{source results} appear in the future (e.g. a construction of $F$'s with high $m(F)$ in degree $5$, which our technique would immediately extend to all multiples of $5$) -- this is contained Corollary~\ref{cor:liftsk}. Our version of the main result is as follows (cf. \cite[Theorem 4]{kalamain}):
\begin{theorem}
	Let $k$, $\ell$ be positive integers and $L$ a totally real number field of degree $\ell$.
    Then for all totally real number fields $K$ of degree $k$ satisfying
    \begin{enumerate}
        \item $\disc_K$ is greater than $C(k,L)$ and coprime to $\disc_L$,
        \item $K$ has no proper subfields apart from $\kve$,
        \item $k\nmid\ell$ or $K/\kve$ is not Galois,
    \end{enumerate}
    we have $m(KL)\geq m(L)$, where $C(k,L)$ is a real constant dependent only on $k$ and $L$. As a consequence, for all $k\geq3$ there are infinitely many totally real number fields $F$ of degree $k\ell$ with $m(F)\geq m(L)$.
\end{theorem}
We will prove this in Section~\ref{sec:main} as Theorem~\ref{thrm:main} and Corollary~\ref{cor:liftsk}.

\section*{Acknowledgments}
I wish to thank Vítězslav Kala for his advice and patience in supervising my bachelor's thesis, on which this article is based.

\section{Preliminaries}
\label{sec:preliminaries}

\subsection{Number fields and quadratic lattices}
\label{subsec:fields-and-lattices}

A~\emph{number field} is a field extension $F$ of the rational numbers $\kve$ of some finite degree $d$ and we denote its rings of algebraic integers as $\Of$. Such an $F$ is equipped with $d$ \emph{embeddings} $\sigma_1,\dots,\sigma_d:F\to\ce$; when all their images lie in $\er$, we say $F$ is \emph{totally real}. In that case we say an $\alpha\in F$ is \emph{totally positive}, if $\sigma(\alpha)>0$ for each $i$, and we denote this by $\alpha\succ0$. More generally, $\alpha\succ\beta$ denotes $\alpha-\beta\succ0$, whilst $\alpha\succeq\beta$ means $\alpha\succ\beta$ or $\alpha=\beta$, and $\Of^+$ denotes the set of totally positive elements of $\Of$.

We equip $F$ with a \emph{trace} and several notions of \emph{discriminant}. The trace is defined as $\Tr_{F/\kve}(\alpha) := \sum_{i=1}^d \sigma_i(\alpha)\in\kve$. We define the discriminant of a $d$-tuple of elements $\alpha_1,\dots,\alpha_d\in F$ as
\[
    \Delta_{F/\kve}(\alpha_1,\dots,\alpha_d) := \det(\Tr_{F/\kve}(\alpha_i\alpha_j))_{i,j=1}^d = \zav{\det(\sigma_i(\alpha_j))_{i,j=1}^d}^2,
\]
and for a single $\alpha\in F$, we define $\Delta_{F/\kve}(\alpha):=\Delta_{F/\kve}(1,\alpha,\dots,\alpha^{d-1})$. If $\omega_1,\dots,\omega_d$ is a $\zet$-basis of $\Of$, we define the discriminant of $F$ as $\disc_F:=\Delta_{F/\kve}(\omega_1,\dots,\omega_d)$; this is independent of the particular choice of basis.

A~\emph{quadratic $\Of$-lattice} (or $\Of$-lattice for short) \emph{of rank $r$} consists of a pair $(\Lambda,Q)$, where $\Lambda$ is a finitely generated $\Of$-submodule of $F^r$ which spans $F^r$ as an $F$-vector space and $Q$ is a quadratic form on $F^r$ which attains values from $\Of$ on $\Lambda$. Over a totally real $F$ we say $(\Lambda,Q)$ is \emph{totally positive}, if $Q$ is totally positive definite, i.e. $Q(v)\succ 0$ for all $v\in F^r\setminus\set0$. An $\Of$-lattice $(\Lambda,Q)$ is said to \emph{represent} an element $\alpha\in \Of$ if $Q(v)=\alpha$ for some $v\in\Lambda$, and a totally positive $\Of$-lattice is \emph{universal} if it represents all elements of $\Of^+$.
For each totally real number field $F$, we denote by $m(F)$ the minimal rank of a totally positive universal $\Of$-lattice.
Note that any quadratic form over $\Of$ may be realized as an $\Of$-lattice by taking $\Lambda$ to be the free module $\Of^r$. In general, the minimal rank of a totally positive universal quadratic {form} may differ from $m(F)$, which considers all {lattices}.

The following theorem generalizes the famous 290-theorem of Bhargava and Hanke \cite{bhargava-hanke} to the context of totally real number fields. It will allow us to pleasantly simplify some technical details in the proof of the main theorem in Section~\ref{sec:main}, since it gives us the assurance that the universality of a totally positive lattice is witnessed by a finite tuple of elements.
\begin{theorem}[{\cite[Corollary 5.8]{chan-oh}}]
	\label{thrm:290}
	For a totally real number field $F$, there exists a finite set $S\subset\Of^+$ such that a totally positive $\Of$-lattice is universal if and only if it represents all elements of $S$.
\end{theorem}
As a trivial corollary, the theorem implies that $m(F)$ is finite for every $F$, since the quadratic form $\sum_{s\in S}sx_s^2$ must be universal, meaning that $m(F)\leq\# S$.

Lastly, let us state the following technical result. It guarantees the existence of totally real number fields with properties that will be useful in Section~\ref{sec:main} for controlling the structure of subfields of composita. For brevity, let us fix the convention that unless stated otherwise, whenever we talk about a field being Galois or its Galois closure, we always mean these over $\kve$.

\begin{theorem}[{\cite[Corollary 1.3]{kedlaya}}]
	\label{thrm:kedlaya}
	Let $k\geq2$ be a positive integer, $S$ a finite set of primes and $C>0$ a real constant. Then there exist infinitely many totally real number fields $K$ of degree $k$ such that:
	\begin{itemize}
		\item $\disc_K$ is not divisible by any $p\in S$,
		\item $\disc_K>C$,
		\item the Galois closure $\tilde K$ of $K$ has the Galois group $\Gal(\tilde K/\kve)\simeq S_k$.
	\end{itemize}
\end{theorem}
Note that this is a special case of the \emph{inverse Galois problem}, which asks which finite groups can occur as Galois groups of finite Galois extensions of a given field. In general, this question is open (see e.g. \cite{jensen-ledet-yui} for more context).

\subsection{Inequalities of traces and discriminants}

A~series of inequalities comparing various traces and discriminants of elements in number fields will be crucial to the argument in Section~\ref{sec:main}. Let us formulate the inequalities we will need here.

\begin{prop}[Cauchy-Schwarz inequality applied to totally real number fields]\label{prop:cauchy}
    Let $(\Lambda,Q)$ be a totally positive quadratic $\Of$-lattice with a corresponding bilinear form $B$ over a totally real number field $F$. Then \[Q(u)Q(v)\succeq B(u,v)^2\] for all $u,v\in\Lambda$. In particular, this implies $\tr_{F/\Q}Q(u)Q(v)\geq\tr_{F/\Q} B(u,v)^2$.
\end{prop}

\begin{prop}[{\cite[§2 II]{schur}}]\label{prop:schur}
    Let $F$ be a totally real number field of degree $n=[F:\kve]>1$ and let $\beta\in F$. Then
    \begin{equation}
        \nonumber
        \tr_{F/\kve}\beta^2 \geq c_n\zav{\Delta_{F/\kve}(\beta)}^{2/(n^2-n)},
    \end{equation}
    where $c_n=\frac{n^2-n}{\zav{1^1\cdot2^2\cdots n^n}^{2/(n^2-n)}}$.
\end{prop}

\begin{prop}[{\cite[III 2.10]{neukirch}}]
    \label{prp:towerdisc}
    If $\kve\subset K\subset L$ are number fields, then $\disc_K^{[L:K]}\mid\disc_L$. This may be weakened to the inequality $\abs{\disc_L}\geq\abs{\disc_K}^{[L:K]}$.
\end{prop}

\section{Subfields of number field composita}
\label{sec:diaggrps}

\subsection{The construction}
\label{subsec:construction}
In our goal to construct number field extensions which do not enable the existence of universal quadratic lattices of smaller rank than in the base field, it will be advantageous to restrict the structure of subfields of such an extension. Particularly, given a starting number field $L$, we will be choosing a suitable number field $K$ and considering their compositum $KL$. For each subfield $L'\subset L$, it is unavoidable that both $L'$ and $KL'$ will be subfields of $KL$, so our goal will be to choose $K$ in such a way that these are in fact the only subfields of $KL$. Stated differently, we wish for each intermediate field $\kve\subset M\subset KL$ to satisfy one of $M\subset L$ or $M\supset K$.

The natural way to examine these properties is of course by Galois correspondence, and in order to facilitate this, we will consider the Galois closures (over $\kve$) of our fields. For reference on basic Galois-theoretic results, which we will use here without a particular attribution, see e.g. \cite[Chapter VI, §1]{lang}.

First, we recall a property of discriminants of composita and in particular Galois closures.

\begin{prop}[\cite{toyama}]
    \label{prop:discradical}
    For number fields $K$, $L$, it holds that \[\disc_{KL}\mid \disc_K^{[KL:K]}\disc_L^{[KL:L]}.\]
\end{prop}

\begin{corollary}\label{galclosureradical}
    Let $\tilde K$ be the Galois closure of a number field $K$. Then a prime number $p$ divides $\disc_{\tilde K}$, if and only if it divides $\disc_K$.
\end{corollary}
\begin{proof}
    If $p\mid \disc_K$, then $\disc_K\mid \disc_{\tilde K}$ due to $K\subset \tilde K$ immediately implies $p\mid\disc_{\tilde K}$.

    On the other hand if $p\mid\disc_{\tilde K}$, consider an $\alpha$ such that $K=\kve(\alpha)$ and let $\alpha_1=\alpha, \alpha_2, \dots, \alpha_n$ be the roots of the minimal polynomial of $\alpha$. Clearly, $\tilde K$ is the compositum of fields $\kve(\alpha_i)$. Proposition~\ref{prop:discradical} implies that a prime may divide the discriminant of a compositum only if the divides one of the discriminants of the individual fields, so we must have $p\mid \disc_{\kve(\alpha_i)}$ for some $i$. But $\kve(\alpha_i)$ is a rupture field of the same irreducible polynomial as $\kve(\alpha)=K$, so in fact $p\mid\disc_{\kve(\alpha_i)}=\disc_K$.
\end{proof}

Now, we translate the property \uv{$M\subset L$ or $M\supset K$ for every $\kve\subset M\subset KL$} into a condition on Galois groups:

\begin{prop}\label{galtranslation}
Let $K$, $L$ be number fields such that $\disc_K$ and $\disc_L$ are coprime and let $\tilde K$, $\tilde L$ be their respective Galois closures. Then $M\subset L$ or $M\supset K$ holds for all intermediate fields $\kve\subset M\subset KL$, if and only if all intermediate groups
\[
    \Gal(\tilde K/K)\times\Gal(\tilde L/L) \leq G\leq \Gal(\tilde K/\kve)\times\Gal(\tilde L/\kve)
\]
satisfy either $G\geq \Gal(\tilde K/\kve)\times \Gal(\tilde L/L)$ or $G\leq \Gal(\tilde K/K)\times\Gal(\tilde L/\kve)$.
\end{prop}
\begin{proof}
By Corollary~\ref{galclosureradical}, we have $\disc_{\tilde K}$, $\disc_{\tilde L}$ coprime, implying that $\tilde K\cap \tilde L = \kve$. Because of this and $\tilde K$, $\tilde L$ being Galois, $\tilde K\tilde L$ is Galois with
\[
    \Gal(\tilde K\tilde L/\kve) \simeq \Gal(\tilde K/\kve)\times\Gal(\tilde L/\kve).
\]
More broadly, we can investigate other Galois groups within the following inclusion diagram of number fields:
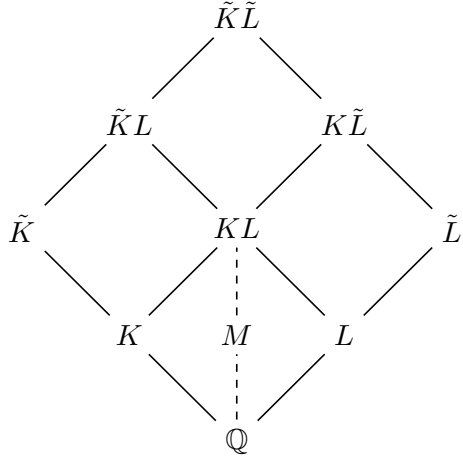
\begin{diagram}[scale=2.0]
    \node[n] (Q) at (0,0){$\kve$};
    \node[n] (K) at (0,1){$K$};
    \node[n] (tK) at (0,2){$\tilde K$};
    \node[n] (L) at (1,0){$L$};
    \node[n] (tL) at (2,0){$\tilde L$};
    \node[n] (KL) at (1,1){$KL$};
    \node[n] (tKL) at (1,2){$\tilde KL$};
    \node[n] (KtL) at (2,1){$K\tilde L$};
    \node[n] (tKtL) at (2,2){$\tilde K\tilde L$};
    \node[n] (M) at (.5,.5){$M$};
    \draw
        (Q) -- (K) -- (tK) -- (tKL) -- (tKtL) -- (KtL) -- (KL) -- (tKL)
        (K) -- (KL) -- (L) -- (tL) -- (KtL)
        (Q) -- (L)
    ;
    \draw[dashed] (Q) -- (M) -- (KL);
\end{diagram}
From $\tilde K/\kve$ being Galois and $\tilde K\cap \tilde L = \tilde K\cap L = \kve$, we get that both $\tilde KL/L$ and $\tilde K\tilde L/\tilde L$ are Galois with
\[
    \Gal(\tilde K/\kve)\simeq\Gal(\tilde KL/L)\simeq\Gal(\tilde K\tilde L/\tilde L).
\]
Analogously, we obtain
\(
    \Gal(\tilde L/\kve)\simeq\Gal(K\tilde L/K)\simeq\Gal(\tilde K\tilde L/\tilde K)
\)
along with these extensions being Galois.

Next, we claim that $\tilde K\cap K\tilde L = K$. One inclusion is obvious. If we briefly denote $X:=\tilde K\cap K\tilde L$, then the equality $X=K$ is proved by bounding
\begin{align*}
[\tilde K\tilde L:\kve] &= [\tilde K(K\tilde L):X]\cdot[X:K]\cdot[K:\kve] \leq [\tilde K:X]\cdot[K\tilde L:X]\cdot[X:K]\cdot[K:\kve] =\\&= [\tilde K:\kve]\cdot[K\tilde L:X] \leq [\tilde K:\kve]\cdot[K\tilde L:K] = [\tilde K:\kve]\cdot[\tilde L:\kve] = [\tilde K\tilde L:\kve]
\end{align*}
and observing equalities must occur.
Now due to $\tilde K\cap K\tilde L = K$ and both $\tilde K/K$ and $K\tilde L/K$ being Galois, we obtain
\[
    \Gal(\tilde K\tilde L/K) \simeq \Gal(\tilde K/K)\times \Gal(K\tilde L/K) \simeq \Gal(\tilde K/K)\times \Gal(\tilde L/\kve).
\]
Analogously, $\Gal(\tilde K\tilde L/L) \simeq \Gal(\tilde K/\kve)\times\Gal(\tilde L/L)$.

Now, since $KL$ is the supremum of $K$ and $L$, by Galois correspondence, $\Gal(\tilde K\tilde L/KL)$ must now be the infimum of $\Gal(\tilde K\tilde L/K)$ and $\Gal(\tilde K\tilde L/L)$, so
\[
    \Gal(\tilde K\tilde L/K) \simeq (\Gal(\tilde K/K)\times \Gal(\tilde L/\kve))\cap (\Gal(\tilde K/\kve)\times\Gal(\tilde L/L)) = \Gal(\tilde K/K)\times\Gal(\tilde L/L).
\]
Note that this also implies $[KL:\kve] = [K:\kve]\cdot[L:\kve]$.

Finally, we apply $\Gal(\tilde K\tilde L/-)$ to the diagram
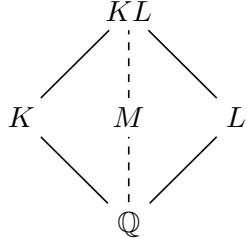
\begin{diagram}[scale=2.0]
    \node[n] (Q) at (0,0){$\kve$};
    \node[n] (K) at (0,1){$K$};
    \node[n] (L) at (1,0){$L$};
    \node[n] (KL) at (1,1){$KL$};
    \node[n] (M) at (.5,.5){$M$};
    \draw (Q) -- (K) -- (KL) -- (L) -- (Q);
    \draw[dashed] (Q) -- (M) -- (KL);
\end{diagram}
and denote $G:=\Gal(\tilde K\tilde L/M)$ to obtain the inclusion diagram of groups:
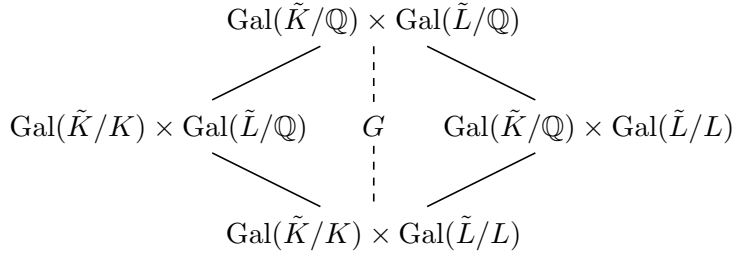
\begin{diagram}[scale=2.0]
    \node[n] (Q) at (1,1){$\Gal(\tilde K/\kve)\times\Gal(\tilde L/\kve)$};
    \node[n] (K) at (-.5,1.5){$\Gal(\tilde K/K)\times\Gal(\tilde L/\kve)$};
    \node[n] (L) at (1.5,-.5){$\Gal(\tilde K/\kve)\times\Gal(\tilde L/L)$};
    \node[n] (KL) at (0,0){$\Gal(\tilde K/K)\times\Gal(\tilde L/L)$};
    \node[n] (M) at (.5,.5){$G$};
    \draw (Q) -- (K) -- (KL) -- (L) -- (Q);
    \draw[dashed] (Q) -- (M) -- (KL);
\end{diagram}
Every $\kve \subset M\subset KL$ corresponds to a $\Gal(\tilde K/\kve)\times\Gal(\tilde L/\kve)\geq G \geq \Gal(\tilde K/K)\times\Gal(\tilde L/L)$ and vice versa. So all $M$'s satisfy $M\subset L$ or $M\supset K$, if and only if all $G$'s satisfy $G\geq \Gal(\tilde K/\kve)\times\Gal(\tilde L/\tilde L)$ or $G\leq \Gal(\tilde K/K)\times\Gal(\tilde L/\kve)$.
\end{proof}

\subsection{Diagonal subgroups in products of groups}

The result of Proposition~\ref{galtranslation} presents a purely group-theoretic question to be explored. Let us a give a name to this property of direct products of groups:
\begin{definition}
Let $S\lneq T$ and $U\lneq V$ be finite groups. Let an intermediate group $G$ with
\[
    S\times U\leq G\leq T\times V
\]
be called \emph{diagonal}, if $G\ngeq T\times U$ and $G\nleq S\times V$. If no such $G$ exists, we will say that $S\times U\leq T\times V$ \emph{admits no diagonal groups}.
\end{definition}
In this language, our interest is in finding, for a given number field $L$, such $K$ that $\Gal(\tilde K/K)\times \Gal(\tilde L/L)\leq \Gal(\tilde K/\kve)\times\Gal(\tilde L/\kve)$ admits no diagonal groups. To investigate this, we will use the following version of Goursat's lemma (cf. \cite[I Ex. 5]{lang}):
\begin{lemma}
    \label{lem:goursat}
    Let $A$ and $B$ be groups. Then there is a bijective correspondence between subgroups $G\leq A\times B$ and quintuples $(A_1,A_2,B_1,B_2,\phi)$, where $A_1$ is a normal subgroup of $A_2$ which is a subgroup of $A$, analogously $B_1$ is a normal subgroup of $B_2$ which is a subgroup of $B$, and $\phi$ is an isomorphism $A_2/A_1\to B_2/B_1$. In this correspondence, $(A_1,A_2,B_1,B_2,\phi)$ maps to
    \[
        G:= \set{(a,b)\in A_2\times B_2\mid \phi(aA_1)=bB_1}.
    \]
\end{lemma}
Observe that in the situation of lemma, $A_2$, $B_2$ are the images of projections $A\times B\to A$, $A\times B\to B$ respectively restricted to $G$, whereas $\set1\times B_1$, $A_1\times\set 1$ are the kernels of these restricted projections. Further, observe that in the special case when $A_1=A_2$, $B_1=B_2$ and $\phi$ is the unique isomorphism between two trivial groups, $G$ is just $A_2\times B_2$.

Returning to our situation with intermediate groups $S\times U\leq G\leq T\times V$,
one may immediately observe that if $S$ is not a maximal subgroup of $T$, then $S\times U\leq T\times V$ admits a diagonal subgroup: if $S\lneq S'\lneq T$, then $S'\times U$ is diagonal -- in Goursat's lemma, this corresponds to the quintuple $(S',S',U,U,\phi)$. The following proposition then shows that this necessary condition of maximality is not too far off from being sufficient:

\begin{prop}\label{maxnonnormal}
Suppose that $S\lneq T$, $U\lneq V$ are groups and that $S$ is a maximal subgroup of $T$ and not a normal subgroup of $T$. Then $S\times U\leq T\times V$ admits no diagonal subgroups.
\end{prop}
\begin{proof}
Suppose that a diagonal $G$ exists and let it correspond to $(A_1,A_2,B_1,B_2,\phi)$ in Goursat's lemma. Then since $S\times U\leq G$, we must have $S\leq A_1$. By maximality of $S$ then, $A_2$ can only be $S$ or $T$. But since $S$ is not normal in $T$, it forces $A_1$, $A_2$ to coincide and to be either $S$ or $T$. This implies that also $B_1=B_2$, hence $G$ is either $S\times B_2$, which is contained in $S\times V$, or $T\times B_2$, which contains $T\times U$, a contradiction.
\end{proof}
\begin{example}
\label{ex:symmetric-groups}
Let us consider the symmetric groups $S_{k-1}\lneq S_k$ for $k\geq 3$, where we interpret $S_{k-1}$ as the subgroup of those permutations that fix the $k$-th element. Then $S_{k-1}$ is a maximal subgroup, since any permutation not fixing $k$ can be turned into any other by composing with elements of $S_{k-1}$. Further, $S_{k-1}$ is not normal e.g. due to $(1\ k)(1\ 2)(1\ k)^{-1} = (2\ k)\notin S_{k-1}$. Thus by the previous Proposition, we see that $S_{k-1}\times U \leq S_k\times V$ admits no diagonal groups for arbitrary $U\lneq V$.
\end{example}

While not constituting a full characterization, we may still extract some necessary conditions for the existence of a diagonal group in the case when $S\lneq T$ is maximal but $S$ is a normal subgroup. In that case, the index $[T:S]$ must be some prime number $p$ and $T/S$ is the cyclic group $C_p$.

\begin{prop}\label{indexdivisibility}
If $S\lneq T$ is normal and $T/S\simeq C_p$ for a prime $p$, then $S\times U\leq T\times V$ admits a diagonal group if and only if there exist $N$, $R$ such that $U\leq N\leq R\leq V$, $N$ is a normal subgroup of $R$ and $R/N\simeq C_p$. In particular, if $p\nmid [V:U]$, a diagonal group cannot exist.
\end{prop}
\begin{proof}
Revisiting the proof of Proposition~\ref{maxnonnormal}, we see that the only new possibility in the quintuple $(A_1,A_2,B_1,B_2,\phi)$ that would correspond to a diagonal group is $A_1=S$, $A_2=T$. After taking $N:=B_1$, $R:=B_2$, the Proposition is then just a rewording of Goursat's lemma.
\end{proof}

As the next two examples demonstrate, $p\mid [V:U]$ is often not sufficient for the existence of a diagonal $G$, but may become sufficient with some additional conditions:
\begin{example}
Consider symmetric groups $U:=S_{p-1}\leq S_p=:V$ for a prime $p\geq3$. Then $[V:U]=p$, but no satisfactory $N$ and $R$ exist: since $S_{p-1}$ is maximal, leaving $N=S_{p-1}$, $R=S_p$ as the only option, which fails because $S_{p-1}$ is not normal.
\end{example}
\begin{example}
In the following cases, $p\mid [V:U]$ alongside an additional property guarantees a diagonal group:
\begin{enumerate}
\item If $U$ is additionally a normal subgroup of $V$, we may choose $N:=U$ and find $R$ as the group corresponding to a cyclic subgroup of order $p$ guaranteed by Cauchy's theorem in $V/U$.
\item If $V$ is nilpotent, then using the normalizer property of nilpotent groups we may construct a sequence $U=N_0 \leq N_1 \cdots\leq N_m = V$ of subgroups, where every $N_{i+1}$ is the normalizer of $N_i$ in $V$; this increasing sequence of subgroups must terminate in $V$ due to $[V:U]$ being finite. Then $p\mid [V:U]$ implies $p\mid [N_{i+1}:N_i]$ for some $i$, so we may perform the construction of (i) on $N_i\leq N_{i+1}$.
\item If $U$ is a $p$-group, let $P$ be a Sylow $p$-subgroup of $V$ that contains $U$. Then we must have $p\mid [P:U]$ and we know as a $p$-group, $P$ is nilpotent, so we may perform the construction of (ii) on $U\leq P$.
\end{enumerate}
\end{example}

We may combine Propositions~\ref{maxnonnormal} and~\ref{indexdivisibility} to state that if $S\lneq T$ is maximal and simultaneously it holds that it is not normal or that $[T:S]\nmid [V:U]$, then $S\times U\leq T\times V$ admits no diagonal groups. Going back to the situation of Proposition~\ref{galtranslation} with $S\lneq T$ playing the role of $\Gal(\tilde K/K)\lneq \Gal(\tilde K/\kve)$ and $U\lneq V$ playing the role of $\Gal(\tilde L/L)\lneq\Gal(\tilde L/\kve)$, this translates to the following:
\begin{prop}\label{compositasummary}
Let $K$, $L$ be number fields of degrees $k$, $\ell$ such that $\disc_K$ and $\disc_L$ are coprime and let $\tilde K$, $\tilde L$ be their respective Galois closures. Then $[KL:\kve]=k\ell$, and if additionally
\begin{enumerate}
\item $K$ has no non-trivial subfields (i.e. no subfields but $\kve$ and $K$ itself), and
\item $K/\kve$ is not Galois or $k\nmid \ell$,
\end{enumerate}
then every intermediate field $\kve\subset M\subset KL$ satisfies either $M\subset L$ or $M\supset K$.
\end{prop}

\section{Extending bounds on the rank of a universal lattice}
\label{sec:main}

Now, we will show how, for a given $L$, to choose a $K$ such that the construction of Section~\ref{sec:diaggrps} yields a compositum $KL$ over which a universal quadratic lattice requires at least as high a rank as would be needed over $L$. Before the main theorem, we prepare a technical lemma concerning a quadratic space viewed over a large field and its subfield.

\begin{lemma}
	\label{lem:KLdim}
	Let $K\subset L$ be totally real number fields and consider an $L$-vector space $U$. Further, let $B: U\times U\to L$ be a symmetric positive definite $L$-bilinear form and $v_1,\dots,v_n\in U$ vectors such that $B(v_i,v_j)\in K$ for each $1\leq i,j\leq n$. Then
	\[
		\dim_K \spn_K\set{v_1,\dots,v_n} = \dim_L \spn_L\set{v_1,\dots,v_n}.
	\]
\end{lemma}
\begin{proof}
    Without loss of generality, let us assume that $v_1,\dots,v_r$ is an $L$-basis of $V_L:=\spn_L\set{v_1,\dots,v_n}$.
    Consider the Gram-Schmidt orthogonalization process performed on $v_1,\dots,v_r$ with respect to $B$, yielding a $B$-orthogonal basis $u_1,\dots,u_r$. This means we will have the relations
    \[
        u_i = v_i - \sum_{j=1}^{i-1} \frac{B(v_i, u_j)}{B(u_j,u_j)}\cdot u_j.
    \]
    Denoting $V_K:=\spn_K\set{v_1,\dots,v_n}$, the condition $B(v_i,v_j)\in K$ implies more generally that $B(V_K,V_K)\subset K$, hence an induction on $i$ shows that $u_i\in V_K$.

    Now we claim that $u_1,\dots,u_r$, in addition to being an $L$-basis of $V_L$, is also a $K$-basis of $V_K$. Since it is linearly independent over $L$, it must be independent over $K$ as well. To prove it generates $V_K$, let $v\in V_K$ be arbitrary. Since $V_K\subset V_L$, we may express it as
    \[
        v~= \sum_{i=1}^r a_i u_i,\quad a_i\in L.
    \]
    Taking $B(-, u_t)$ on both sides for $t=1,\dots,r$, we get $a_t = \frac{B(v,u_t)}{B(u_t,u_t)}\in K$.
    Thus $u_1, \dots, u_r$ is also a $K$-basis of $V_K$, meaning $\dim_K V_K = r = \dim_L V_L$.
\end{proof}

Now we can prove the main theorem:
\begin{theorem}
	\label{thrm:main}
	Let $k$, $\ell$ be positive integers and $L$ a totally real number field of degree $\ell$. Then there exists a real constant $C=C(L,k)$ dependent only on $L$ and $k$ such that $m(KL)\geq m(L)$ for all totally real number fields $K$ of degree $k$ satisfying the following four conditions:
	\begin{enumerate}
		\item $\disc_K>C$,
		\item $\disc_K$ is coprime to $\disc_L$,
		\item $K$ has no non-trivial subfields,
		\item $K$ is not Galois or $k\nmid\ell$.
	\end{enumerate}
\end{theorem}
\begin{proof}
	By Theorem~\ref{thrm:290}, we can choose elements $a_1,\dots,a_n\in \Ol^+$ such that a totally positive $\Ol$-lattice is universal, if and only if it represents each $a_i$. From (ii)--(iv), Proposition~\ref{compositasummary} tells us that $KL$ has degree $k\ell$ and its subfields $M$ necessarily have either $M\subset L$ or $M\supset K$.

	Let $(\Lambda, Q)$ be a totally positive universal $\mathcal O_{KL}$-lattice of rank $r$. Then $Q$ must represent each $a_i$, so we can choose vectors $v_i\in \Lambda$ such that $Q(v_i)=a_i$. Let us also consider the symmetric bilinear form $B$ corresponding to $Q$ and let us denote $\frac{b_{ij}}2:=B(v_i,v_j)$. Note that
	\[
		b_{ij} = 2B(v_i,v_j) = Q(v_i+v_j)-Q(v_i)-Q(v_j)\in\mathcal O_{KL}.
	\]
	We will now show that if $b_{ij}\notin L$ for some $i$, $j$, then $\disc_K$ is bounded. For this, we consider the field $M:=\kve(b_{ij})$. If we have $b_{ij}\notin L$, then $M\not\subset L$ means $M\supset K$; let $e:=[M:K]\mid [KL:K]=\ell$. We will utilize inequalities of traces and discriminants. Since the individual $a_i$ and $b_{ij}$ are defined using $Q$ and $B$, Proposition~\ref{prop:cauchy} yields $4a_ia_j\succeq b_{ij}^2$ and thus
	\[
		\tr_{KL/\kve}(4a_ia_j)\geq\tr_{KL/\kve}b_{ij}^2 = \frac\ell e\tr_{M/\kve}b_{ij}^2.
	\]
	Next, Proposition~\ref{prop:schur} bounds $\tr_{M/\kve}b_{ij}^2\geq c_{ke}\zav{\Delta_{M/\kve}(b_{ij})}^{2/((ke)^2-ke)}$, where $c_x = \frac{x^2-x}{\zav{1^1\cdot2^2\cdots x^x}^{2/(x^2-x)}}$ only depends on $x$. Since $b_{ij}$ generates $M$ as a field extension of $\kve$ and it is an element of $\mathcal O_{KL}\cap M = \mathcal O_M$, its discriminant must be a non-zero multiple of $\disc_M$. Therefore, we bound
	\[
		\Delta_{M/\kve}(b_{ij}) \geq \disc_M \geq \disc_K^e,
	\]
	where the last inequality is Proposition~\ref{prp:towerdisc}.
	Altogether, from $b_{ij}$ not belonging to $L$ we obtain the bound
	\begin{align*}
		k\tr_{L/\kve}(4a_ia_j) &= \tr_{KL/\kve}(4a_ia_j) \geq \frac{\ell c_{ke}}e\zav{\disc_K^e}^{2/((ke)^2-ke)} = \frac{\ell c_{ke}}e \disc_K^{2/(k^2e-k)},
	\end{align*}
	which rearranges to $\disc_K\leq \zav{\frac{ke\tr_{L/\kve}(4a_ia_j)}{\ell c_{ke}}}^{\frac{k^2e-k}2}$. This depends on $i$, $j$ and on the extension degree $e\mid\ell$, both of which admit only finitely many options, so maximizing over them, we may denote $T:=\max_{1\leq i<j\leq n}\tr_{L/\kve}(4a_ia_j)$ and obtain the overall bound
	\[
		\disc_K \leq \max_{e\mid\ell} \zav{\frac{keT}{\ell c_{ke}}}^{\frac{k^2e-k}2} =: C(L,k) = C.
	\]

	This inequality means that if $\disc_K > C$, each $b_{ij}$ must necessarily belong to $L$. This will mean that
	\[
		B(v_i,v_j) = \left\{\begin{array}{ll}
			a_i, & \text{if $i=j$,}\\
			\frac{b_{ij}}2, & \text{if $i\neq j$}
		\end{array}\right\}\in L
	\]
	for all $1\leq i,j\leq n$. This allows us to use Lemma~\ref{lem:KLdim} to obtain
	\[
		r' := \dim_L\spn_L\set{v_1,\dots,v_n} = \dim_{KL}\spn_{KL}\set{v_1,\dots,v_n} \leq r,
	\]
	since $v_1,\dots,v_n$ belong to $\Lambda\subset(KL)^r$. Now we can identify $\spn_L\set{v_1,\dots,v_n}$ with $L^{r'}$ and consider
	\begin{align*}
		\Lambda' := \spn_{\Ol}\set{v_1,\dots,v_n}.
	\end{align*}
	Since $B$ takes values from $L\cap \mathcal{O}_{KL} = \Ol$ on pairs of vectors that generate $\Lambda'$ as an $\Ol$-module, it must take values from $\Ol$ on the entire $\Lambda'$. Thus when we restrict
	\[
		Q' := \restr Q{\Lambda'},
	\]
	we obtain a quadratic $\Ol$-lattice $(\Lambda',Q')$ of rank $r'$. It has to be totally positive, since it is just a restriction of the totally positive $Q$. Finally, $Q'$ represents each $a_i$, so from the choice of $a_1,\dots,a_n$, it is a universal $\Ol$-lattice. Thus we obtain
	\(
		m(L)\leq r'\leq r,
	\)
	meaning a universal $\mathcal O_{KL}$-lattice has to have rank at least $m(L)$, and so $m(KL)\geq m(L)$.
\end{proof}

\begin{corollary}
	\label{cor:liftsk}
	Let $L$ be a totally real number field of degree $\ell$ and let $k\geq3$ be an integer. Then there exist infinitely many totally real number fields $F\supset L$ of degree $k\ell$ that have $m(F)\geq m(L)$.
\end{corollary}
\begin{proof}
	For the sake of a contradiction, suppose only finitely many such $F$ exist. Let $D$ be the maximum of their discriminants and let $C=C(L,k)$ be the constant from Theorem~\ref{thrm:main}. From Theorem~\ref{thrm:kedlaya}, there are infinitely many totally real number fields $K$ of degree $k$ such that $\disc_K>\max(C,D)$ is coprime to $\disc_L$ and whose Galois closure $\tilde K$ has $\Gal(\tilde K/\kve)\simeq S_k$.

    It is easily seen that then $\Gal(\tilde K/K)\simeq S_{k-1}$ and that this $S_{k-1}$ is embedded in $\Gal(\tilde K/\kve)\simeq S_k$ as the subgroup of permutations fixing one chosen element. From Example~\ref{ex:symmetric-groups}, we then know $\Gal(\tilde K/K)$ is a maximal subgroup of $\Gal(\tilde K/\kve)$ and that it is not normal. Thus the conditions of Theorem~\ref{thrm:main} are satisfied and $F:=KL$ has $m(F)\geq m(L)$. But clearly $\disc_F\geq \disc_K>D$, which contradicts the choice of $D$. Thus there must be an infinite amount of fields $F$ that satisfy the statement of the corollary.
\end{proof}

Using the corollary, we can recover a simpler proof of \cite[Theorem 1]{kalamain}. This is because our Theorem~\ref{thrm:main} contains fewer technical conditions compared to \cite[Theorem 4]{kalamain}, eliminating some of the need for case splitting in the following proof:

\begin{theorem}
	\label{thrm:23deg}
	Let $d$, $r$ be positive integers such that $d$ is divisible by $2$ or $3$. Then there exist infinitely many totally real number fields $F$ of degree $d$ with $m(F)\geq r$.
\end{theorem}
\begin{proof}
	Let us say that a positive integer $d$ is \emph{suitable} if for arbitrary positive $r$, there exist infinitely many totally real number fields $F$ of degree $d$ with $m(F)\geq r$. Then \cite[Theorem 1.1]{kala-yatsyna-zmija} implies that $2$ is suitable and \cite[Theorem 1.1]{man} improves this to all powers of $2$ being suitable (these results in fact bound $m(F)$ from below by a function of $\disc_F$ for almost all quadratic and multiquadratic fields respectively). In a similar vein, $3$ is suitable by \cite[Theorem 1.5]{yatsyna} (though it may be formulated only for quadratic forms, the proof works for quadratic lattices without major alterations).

	Corollary~\ref{cor:liftsk} implies that whenever $d$ is suitable, then so is $kd$ for all $k\geq3$. Then we easily see that any multiple of $2$ or $3$ can be achieved as $kd$, $k\geq3$ with one of $d=2$, $3$ or $4$, proving the theorem.
\end{proof}
More generally, the above technique can be seen as reducing the question of whether there exist (infinitely many) totally real number fields $F$ of degree $d$ with $m(F)\geq r$ to prime degrees $d$. This however remains open for primes $\geq5$.

Lastly, it would be interesting to be able to construct totally real number fields $F$ with large $m(F)$, a given degree $d$ and additionally a given Galois structure, i.e. prescribed Galois groups $\Gal(\tilde F/\kve)$ and $\Gal(\tilde F/F)$. In this sense, the construction of Theorem~\ref{thrm:main} is quite limited, since we only obtain Galois groups that are direct products of smaller Galois groups. We should note that due to the inverse Galois problem being open, it is not certain that a prescribed $\Gal(\tilde F/\kve)$ can always be achieved at all, even without the additional demand on $m(F)$.


\begin{thebibliography}{KTZ}

\bibitem[BH]{bhargava-hanke}
    {\sc M. Bhargava, J. Hanke,}
    \emph{Universal quadratic forms and the 290-Theorem,}
    preprint.

\bibitem[BK]{blomer-kala}
    {\sc V. Blomer, V. Kala,}
    \emph{Number fields without universal $n$-ary quadratic forms},
    Math. Proc. Cambridge Philos. Soc.
    159 (2015),
    239--252.

\bibitem[CL+]{cech-et-al}
    {\sc M. Čech, D. Lachman, J. Svoboda, M. Tinková, K. Zemková,}
    \emph{Universal quadratic forms and indecomposables over biquadratic fields,}
    Math. Nachr.
    292 (2019),
    540--555.

\bibitem[CO]{chan-oh}
    {\sc W. K. Chan, B.-K. Oh,}
    \emph{Can we recover an integral quadratic form by representing all its subforms?,}
    Adv. Math.
    433 (2023)
    109317.

\bibitem[Ch]{cohn}
    {\sc H. Cohn,}
    \emph{Decomposition into four integral squares in the fields of $2^{1/2}$ and $3^{1/2}$,}
    Amer. J. Math.,
    82 (1960),
    301--322.

\bibitem[Go]{gotzky}
  {\sc F. G\"otzky,}
  \emph{Über eine zahlentheoretische Anwendung von Modulfunktionen zweier Veränderlicher,}
  Math. Ann. 100 (1928),
  411--437.

\bibitem[JLY]{jensen-ledet-yui}
    {\sc C. U. Jensen, A. Ledet, N. Yui,}
    \emph{Generic Polynomials: Constructive Aspects of the Inverse Galois Problem},
    Cambridge University Press,
    2002.

\bibitem[Ka1]{kalaquad}
    {\sc V. Kala,}
    \emph{Universal quadratic forms and elements of small norm in real quadratic fields,}
    Bull. Aust. Math. Soc.
    94 (2016),
    7--14.

\bibitem[Ka2]{kalamain}
    {\sc V. Kala,}
    \emph{Number fields without universal quadratic forms of small rank exist in most degrees,}
    Math. Proc. Cambridge Philos. Soc.
    174 (2023),
    225--231.

\bibitem[Ka3]{kala-survey}
  {\sc V. Kala},
  \emph{Universal quadratic forms and indecomposables in number fields: A~survey,}
  Commun. Math. 31 (2023),
  81--114.

\bibitem[KS]{kala-svoboda}
    {\sc V. Kala, J. Svoboda,}
    \emph{Universal quadratic forms over multiquadratic fields,}
    Ramanujan J.
    48 (2019),
    151--157.

\bibitem[KT]{kala-tinkova}
    {\sc V. Kala, M. Tinková,}
    \emph{Universal quadratic forms, small norms and traces in families of number fields,}
    Int. Math. Res. Not. IMRN
    (2023),
    7541--7577.

\bibitem[KY1]{kala-yatsyna1}
  {\sc V. Kala, P. Yatsyna},
  \textit{Lifting problem for universal quadratic forms},
  Adv. Math. 377 (2021),
  107497, 24 pp.

\bibitem[KY2]{kala-yatsyna2}
  {\sc V. Kala, P. Yatsyna},
  \textit{On Kitaoka's conjecture and lifting problem for universal quadratic forms},
  Bull. Lond. Math. Soc.
  55 (2023),
  854--864.

\bibitem[KYZ]{kala-yatsyna-zmija}
    {\sc V. Kala, P. Yatsyna, B. \.Zmija,}
    \emph{Real quadratic fields with a universal form of given rank have density zero,}
    Amer. J. Math. (to appear),
    18. pp.

\bibitem[Ke]{kedlaya}
    {\sc K. S. Kedlaya,}
    \emph{A~construction of polynomials with squarefree discriminants,}
    Proc. Amer. Math. Soc.
    140 (2012),
    3025--3033.

\bibitem[Ki]{kirmse}
  {\sc J. Kirmse},
  \emph{Zur Darstellung total positiver Zahlen als Summen von vier Quadraten},
  Math. Z. 21 (1924),
  195--202.

\bibitem[KTZ]{krasensky-tinkova-zemkova}
    {\sc J. Krásenský, M. Tinková, K. Zemková,}
    \emph{There are no universal ternary quadratic forms over biquadratic fields,}
    Proc. Edinb. Math. Soc.
    63 (2020),
    861--912.

\bibitem[La]{lang}
    {\sc S. Lang,}
    \emph{Algebra},
    Springer-Verlag,
    New York (2002).

\bibitem[Ma]{maass}
  {\sc H. Maaß,}
  \emph{Über die Darstellung total positiver Zahlen des Körpers $R(\sqrt5)$ als Summe von drei Quadraten,}
  Abh. Math. Semin. Univ. Hambg.
  14 (1941),
  185--191.

\bibitem[Mn]{man}
    {\sc S. H. Man,}
    \emph{Minimal rank of universal lattices and number of indecomposable elements in real multiquadratic fields,}
    Adv. Math.
    447 (2024),
    Article 109694,
    38 pp.

\bibitem[Neu]{neukirch}
  {\sc J. Neukirch,}
  \emph{Algebraic Number Theory},
  Springer-Verlag,
  Berlin (1999).

\bibitem[Sch]{schur}
  {\sc I. Schur,}
  \emph{{Ü}ber die {V}erteilung der {W}urzeln bei gewissen algebraischen Glei\-chungen mit ganzzahligen Koeffizienten},
  Math. Z.
  1 (1918),
  377--402.

\bibitem[Si]{siegel}
  {\sc C. L. Siegel,}
  \emph{Sums of $m$-th powers of algebraic integers,}
  Ann. of Math.
  46 (1945),
  313--339.

\bibitem[To]{toyama}
  {\sc H. T\^oyama,}
  \emph{A~note on the different of the composed field},
  Kodai Math. Sem. Rep.
  7 (1955),
  43--44.

\bibitem[Ya]{yatsyna}
  {\sc P. Yatsyna,}
  \emph{A~lower bound for the rank of a universal quadratic form with integer coefficients in a totally real field,}
  Comment. Math. Helvet.
  94 (2019),
  221--239.

\end{thebibliography}
\end{document}